\documentclass[preprint, 12pt, sort&compress]{elsarticle}

\usepackage{latexsym}
\usepackage{amsfonts}
\usepackage{graphicx,amssymb,natbib}
\usepackage{amsmath}
\usepackage{amssymb}
\usepackage[mathscr]{eucal}
\usepackage{enumerate}
\usepackage{amsthm}

\usepackage{bm}

\usepackage{color, soul} 

\begin{document}

\newtheorem{tm}{Theorem}[section]
\newtheorem{pp}[tm]{Proposition}
\newtheorem{lm}[tm]{Lemma}
\newtheorem{df}[tm]{Definition}
\newtheorem{tl}[tm]{Corollary}
\newtheorem{re}[tm]{Remark}
\newtheorem{eap}[tm]{Example}

\newcommand{\pof}{\noindent {\bf Proof} }
\newcommand{\ep}{$\quad \Box$}

\newcommand{\al}{\alpha}
\newcommand{\be}{\beta}
\newcommand{\var}{\varepsilon}
\newcommand{\la}{\lambda}
\newcommand{\de}{\delta}
\newcommand{\str}{\stackrel}

\renewcommand{\proofname}{\bf Proof}

\allowdisplaybreaks

\begin{frontmatter}

\title{Some properties of Skorokhod metric on fuzzy sets
 }
\author{Huan Huang}
 \ead{hhuangjy@126.com }
\address{Department of Mathematics, Jimei
University, Xiamen 361021, China}


\date{}

\begin{abstract}

In this paper, we have our discussions on normal and upper semi-continuous fuzzy sets on metric spaces.
The Skorokhod-type metric is stronger than the Skorokhod metric.
  It
is
found that
 the Skorokhod metric and the Skorokhod-type metric are equivalent on compact fuzzy sets.
However,
 the Skorokhod metric and the Skorokhod-type metric need not be equivalent
on
$L_p$-integrable fuzzy sets.
Based on this,
we
investigate relations between these two metrics
and
the $L_p$-type $d_p$ metric.
It is found
that
the relations can be divided into
three cases.
On compact fuzzy sets, the Skorokhod metric is stronger than the $d_p$ metric.
On $L_p$-integrable fuzzy sets, which take compact fuzzy sets as special cases,
the Skorokhod metric is not necessarily stronger than the $d_p$ metric,
but the Skorokhod-type metric is still stronger than
the $d_p$ metric.
On general fuzzy sets, even the Skorokhod-type metric is not necessarily stronger than
the $d_p$ metric.
We also show that
the Skorokhod metric is stronger than
the sendograph metric.
\end{abstract}


\end{frontmatter}

\section{Introduction}

Skorokhod metric on fuzzy sets has received deserving attentions.
Joo and Kim \cite{joo} introduced
the Skorokhod metric and the Skorokhod-type metric.
They \cite{joo,joo2}
have proven that
the Skorokhod metric and the Skorokhod-type metric are equivalent.
Joo and Kim \cite{joo2}
have pointed out
that
the Skorokhod metric is stronger than the $d_p$ metric.
Kim and Kim \cite{kim} have proven
that
the Skorokhod metric is stronger than the sendograph metric.
These results are
obtained
 on the set of normal, upper semi-continuous
and
compact fuzzy sets on $\mathbb{R}^m$.
Recently,
Jard\'{o}n, S\'{a}nchez and Sanchis
\cite{jarn}
discussed the Skorokhod metric on compact fuzzy sets on metric space rather than just on $\mathbb{R}^m$.
Here, a compact fuzzy set is a fuzzy set whose support set is compact, and a noncompact fuzzy set is a general fuzzy set whose support set may be compact or non-compact.

The $L_p$-type $d_p$ metric,
endograph metric and sendograph metric
are
important and widely used metrics on fuzzy sets \cite{da, du, kloeden, wu, kupka, grzegorzewski, grzegorzewski2, huang, huang9}.
 Noncompact fuzzy sets, which take compact fuzzy sets as special cases, have received considerable attention
from points of view of theory and practical applications \cite{kloeden, huang, grzegorzewski}.
The Euclidean space $\mathbb{R}^m$
is
a kind of metric space.
Of course, it is worthy to study fuzzy sets on general metric space.

So,
it is important and natural to consider relation of the Skorokhod metric and the Skorokhod-type metric,
and
relation of these two metrics
with
 the $d_p$ metric,
the endograph metric and the sendograph metric, respectively,
on noncompact fuzzy sets on metric space.
 In
 this paper, the discussions are carried out on normal and upper semi-continuous fuzzy sets on metric space.
We
assume that all the fuzzy sets mentioned in this paper are fuzzy sets of this type.

We confirm
 that
 the
 relations of the Skorokhod metric, the Skorokhod-type metric and the $d_p$ metric
  obtained in \cite{joo,joo2} still hold
  on compact fuzzy sets.

 However, we find that the
 relations of the Skorokhod metric, the Skorokhod-type metric and the $d_p$ metric
 on noncompact fuzzy sets
  are
   quite different from the case of compact fuzzy sets. This is the focus of our paper.

   We show
   that
   the Skorokhod metric is stronger than the sendograph metric and the endograph metric.

The remainder of this paper is organized as follows.
In Section 2, we recall some basic notions about fuzzy sets and various type of metrics on fuzzy sets.
Then
we
 introduce some subclasses of fuzzy sets and have
some
discussions
related to metrics on fuzzy sets, which are useful in the sequel of this paper.
In Section 3, we discuss the relation
of
the Skorokhod metric and the Skorokhod-type metric.
In Sections 4, 5 and 6, we investigate relations between the two metrics, the Skorokhod metric and the Skorokhod-type metric,
and
the $d_p$ metric.
The relation is divided into three cases.
 We
 mainly discuss one case in each section.
In Section 7, we consider relation of the Skorokhod metric and the two metrics, the sendograph metric and the endograph metric.
In Section 8, we give a simple example to answer some questions discussed recently.
  At last, we draw our conclusions in Section 9.

\section{Fuzzy sets and metrics on them} \label{fsm}

In this section, we recall basic notions about fuzzy sets and metrics on fuzzy sets.
Readers
can refer to \cite{wu, da} for more contents.
We
also introduce some subclasses of fuzzy sets and have
some
discussions
 related to metrics on fuzzy sets.

Let $(X,d)$ be a metric space and let $K(X)$ and
 $C(X)$
 denote the set of all non-empty compact subsets of $X$ and the set of all non-empty closed subsets of $X$, respectively.

Metric $d_a$ is said to be stronger than metric $d_b$ on $X$, if the metric $d_a$ convergence can imply the metric $d_b$ convergence
on
$X$.
Two
metrics $d_a$ and $d_b$ are called to be equivalent on $X$ iff $d_a$ is stronger than $d_b$ on $X$ and $d_b$ is stronger than $d_a$ on $X$.

Let
$F(X)$
denote the set of
all fuzzy sets on $X$. A fuzzy set $u\in F(X)$ can be seen as a function $u:X \to [0,1]$.
In this sense,
a
subset $S$ of $X$ can be seen as a fuzzy set
\[\widehat{S}(x) = \left\{
                    \begin{array}{ll}
                      1, & x\in S, \\
                      0, & x\not\in S.
                    \end{array}
                  \right.
\]

For
$u\in F(X)$, let $[u]_{\al}$ denote the $\al$-cut of
$u$, i.e.
\[
[u]_{\al}=\begin{cases}
\{x\in X : u(x)\geq \al \}, & \ \al\in(0,1],
\\
{\rm supp}\, u=\overline{    \{ u > 0 \}    }, & \ \al=0.
\end{cases}
\]

For
$u\in F(X)$,
define
\begin{gather*}
{\rm end}\, u:= \{ (x, t)\in  X \times [0,1]: u(x) \geq t\},
\\
{\rm send}\, u:= \{ (x, t)\in  X \times [0,1]: u(x) \geq t\} \cap  ([u]_0 \times [0,1]).
\end{gather*}
 $
{\rm end}\, u$ and ${\rm send}\, u$
 are called the endograph and the sendograph of $u$, respectively. The metric $\overline{d}$ on $X \times [0,1]$ is defined
as
$$  \overline{d } ((x,\al), (y, \beta)) = d(x,y) + |\al-\beta| .$$

Let
$F_{USC}^1(X)$
denote
the set of all normal and upper semi-continuous fuzzy sets $u:X \to [0,1]$,
i.e.,
$$F_{USC}^1(X) :=\{ u\in F(X) : [u]_\al \in  C(X)  \  \mbox{for all} \   \al \in [0,1]   \}.  $$

 We
use $H$ to denote the Hausdorff metric
on
 $C(X)$ induced by $d$, i.e.,
$$
H(U,V)=\max\{H^{*}(U,V),\ H^{*}(V,U)\}
$$
for arbitrary $U,V\in C(X)$,
where
  $$
H^{*}(U,V)=\sup\limits_{u\in U}\,d\, (u,V) =\sup\limits_{u\in U}\inf\limits_{v\in
V}d\, (u,v).
$$
If there is no confusion, we also use $H$ to denote the Hausdorff metric on $C(X\times [0,1])$ induced by $\overline{d}$.

The endograph metric $H_{\rm end} $, the sendograph metric $H_{\rm send} $,
the supremum metric $d_\infty$, the Skorokhod metric $\rho_0$ and the Skorokhod-type metric
$\rho_1$
can be defined on $F_{USC}^1(X)$ as usual.
The readers can
see
\cite{huang,da,jarn,joo,wug,wu,kloeden} for related contents.

For $u,v \in F_{USC}^1(X)$,
\begin{gather*}
H_{\rm end}(u,v): =  H({\rm end}\, u,  {\rm end}\, v ),
\\
H_{\rm send}(u,v): =  H({\rm send}\, u,  {\rm send}\, v ),
\\
d_\infty (u,v): =  \sup \{H([u]_\al, [v]_\al): \al\in [0,1]\},
\\
\rho_0(u,v):= \inf \{\varepsilon:  \mbox{there  exists a} \ t \  \mbox{in} \ T
\
\mbox{such that}
\ D(t) \leq \varepsilon \ \mbox{and} \ d_\infty (u,tv)\leq \varepsilon \},
\\
\rho_1(u,v):= \inf \{\varepsilon:  \mbox{there  exists a} \ t \  \mbox{in} \ T
\
\mbox{such that}
\ S(t) \leq \varepsilon \ \mbox{and} \ d_\infty (u,tv)\leq \varepsilon \}
  \end{gather*}
where
 $T$ is the class of strictly increasing, continuous
mapping of $[0,1]$ onto $[0,1]$,
\begin{gather*}
D(t): = \sup\{|t(\al)-\alpha|:\al\in[0,1]\}, \    \mbox{and}
\\
S(t): = \sup\{ | \ln\frac{t(\beta) - t(\al)}{\beta-\al} |    :\al \not= \beta,\  \al,\beta \in[0,1]  \}
\end{gather*}
for each $t\in T$.

\begin{re}
  {\rm

In \cite{joo}, $\rho_0$ and $\rho_1$ are written as $d_0$ and $d_1$, respectively.
It's
 also mentioned in \cite{joo}
that
 $D(t) = D(t^{-1})$
and $S(t) = S(t^{-1})$ for all $t\in T$. In this paper, we call $\rho_1$ the ``Skorokhod-type'' metric.

 }
\end{re}

\begin{re}\label{ser}
{\rm
It can be checked that for $u$, $u_n$, $n=1,2,\ldots$, in $F^1_{USC} (X)$,
$H_{\rm send} (u_n, u) \to 0$ is equivalent to $  H_{\rm end} (u_n, u) \to 0 $ and  $ H([u_n]_0, [u]_0) \to 0$.
  }
\end{re}

The $d_p$ metrics, $p \geq 1$, are widely used   $L_p$-type metrics on fuzzy set spaces,
which
are very important from points of view of theory and applications.

The $d_p$ metric
can
be defined    on $ F^1_{USC} (X)$
as usual, i.e.,
$$d_p(u,v) :=   \left(\int_0^1 H([u]_\al, [v]_\al   )^p  \,   d\al   \right)^{1/p} $$
for
 $u,v \in  F^1_{USC} (X)$ when $H([u]_\bullet, [v]_\bullet)$ is a measurable function on $[0,1]$.

\begin{df} We introduce the following subsets of $ F^1_{USC}(X)$, which will be useful in the sequel.

$F^1_{USCB}(X):=\{ u\in  F^1_{USC}(X): [u]_0 \in K(X) \}$.

$F^1_{USCG}(X):=\{ u\in  F^1_{USC}(X): [u]_\al \in K(X) \ \mbox{when} \   \al>0 \}$.

$F^1_{USCG} (X)^p  :=  \{ u\in  F^1_{USCG}(X):   d_p(u, \widehat{\{x_0\}} ) = (\int_0^1 H([u]_\al, \{x_0\}   )^p  \,   d\al)^{1/p} <  +\infty     \}$,
where $p\geq 1$ and $x_0$ is a point in $X$.

\end{df}
 The definition of $F^1_{USCG} (X)^p$ does not depend on the choice of $x_0$.
Clearly,
 $$F^1_{USCB}(X) \subset  F^1_{USCG} (X)^p   \subset  F^1_{USCG}(X)   \subset  F^1_{USC}(X).$$
Let $t\in T$. Then
$tu \in F^1_{USCB}(X)$
iff $u\in F^1_{USCB}(X)$, and $u\in F^1_{USCG}(X)$
iff $tu \in F^1_{USCG}(X)$.

Let
 $\mathbb{R}^m$, $m\geq 1$,
be the $m$-dimensional Euclidean space. Usually, we write $\mathbb{R}^1$ as $\mathbb{R}$ for simplicity.
It can be checked that the following statements hold.
\begin{itemize}
  \item  For $u \in  F^1_{USC}(X)$ and $x_0 \in X$,  $H([u]_\bullet, \{x_0\}) $ is a measurable function on $[0,1]$.

  \item For $u,v \in  F^1_{USCG}(X)$, $H([u]_\bullet, [v]_\bullet) $ is a measurable function on $[0,1]$.

\item  For    $u,v \in  F^1_{USC} (\mathbb{R}^m)$, $H([u]_\bullet, [v]_\bullet) $ is a measurable function on $[0,1]$.

\item  There exists metric space $X$ and $u,v\in F^1_{USC}(X)$ such that $H([u]_\bullet, [v]_\bullet) $ is a non-measurable function
on
$[0,1]$.
\end{itemize}

\begin{re}

In \cite{jarn}, the symbol
$\mathcal{F}(X)$
is used to denote $F_{USCB}^1(X)$.

\end{re}

The set of (compact) fuzzy numbers
are
denoted by
$E^m$. It
is
defined as
$$
E^m:=
\{ u\in F(\mathbb{R}^m):  [u]_\al  \mbox{ is a nonempty compact convex subset of } \mathbb{R}^m \mbox{ for } \al\in [0,1] \}.
$$
Fuzzy numbers have attracted much attention from theoretical research and practical applications \cite{da}.

For $u,v \in F^1_{USC}(X)$,  $H([u]_\bullet, [v]_\bullet)$ could be a non-measurable function on $[0,1]$.
So
we introduce the following $d_p^*$ distance on $F^1_{USC} (X)$.

The
 $d_p^*$ distance, $p\geq 1$, on $F^1_{USC} (X)$,
 is defined as
$$
d_p^*(u,v) := \inf \left\{    \left(\int_0^1 f(\al)^p  \,   d\al   \right)^{1/p}
 :
f \mbox{ is a measurable function on}\ [0,1] \mbox{with} \  f(\cdot)  \geq H([u]_\bullet, [v]_\bullet)     \right\}
$$
for
$u,v \in F^1_{USC} (X)$.

\begin{tm} \label{dpsm}
  $d_p^*$ is a metric on $F^1_{USC} (X)$.
\end{tm}

\begin{proof}
  See Appendix.
\end{proof}

\begin{re}{\rm
 Clearly, $d_p^*(u,v)= d_p(u,v)$   when $H([u]_\bullet, [v]_\bullet)$ is a measurable function on $[0,1]$.
So
 $d_p^*$ metric is an expansion of the $d_p$ metric on $F^1_{USC} (X)$.

In the sequel, we don't distinguish between $d_p^*$ and $d_p$, both of which are written as $d_p$.
}
\end{re}

\section{Relation between   $\rho_0$ and $\rho_1$ on $F^1_{USC}(X)$}

Joo and Kim \cite{joo} have proven
that
the Skorokhod metric $\rho_0$ and the Skorokhod-type metric $\rho_1$ are equivalent on
$E$.
Further,
they \cite{joo2} pointed out
that
this relation still holds on $F^1_{USCB}(\mathbb{R}^m)$.

In this section, we discuss the relation of the Skorokhod metric $\rho_0$ and the Skorokhod-type metric $\rho_1$ on $F^1_{USC}(X)$.
The $\rho_1$ convergence can still imply the $\rho_0$ convergence on $F^1_{USC} (X)$. This statement can be deduced in the same way as the corresponding conclusion on $E$ in \cite{joo}.
By establishing some lemmas,
we point out that the method in \cite{joo} can also be used to prove
the equivalence
of
 $\rho_0$ and $\rho_1$
on
$F^1_{USCB} (X)$.
However,
we find
that
these two metrics are not necessarily equivalent on $F^1_{USCG} (X)^p$, which is larger than $F^1_{USCB} (X)$.
A
counterexample is given to
 show that
the $\rho_0$ convergence need not imply the $\rho_1$ convergence on $F^1_{USCG} (\mathbb{R})^p$.

\begin{lm} \label{der}

(\romannumeral1) $D(t) \to 0$ need not imply $S(t) \to 0$.

(\romannumeral2) $S(t) \to 0$ implies that $D(t) \to 0$.

\end{lm}

\begin{proof}
  (\romannumeral1)
Example \ref{tgf} provides a counterexample to show this statement.

(\romannumeral2)
This statement is from \cite{joo}.
In fact,
 it
can
 be deduced from the proof of Lemma 3.5 in \cite{joo}.

\end{proof}

\begin{eap} \label{tgf}
 {\rm
Consider
$t_a \in T$, $a\in (0,1)$, defined as
\[
 t_a(\xi)=\left\{
            \begin{array}{ll}
              \sqrt{\xi}, & \xi\in [0,a], \\
              \dfrac{  1-\sqrt{a}     }    {1-a} \xi +  \dfrac{   \sqrt{a} - a  }  {1-a}, &  \xi \in [a, 1].
            \end{array}
          \right.
\]
It
can be checked that $D(t_a)= |a- \sqrt{a}|$ for $a\leq \frac{1}{4}$ and $S(t_a) \equiv +\infty$ for each $a\in (0,1)$.
Thus
$D(t_a) \to  0$ as $a\to 0$. However $S(t_a) \not \to 0$ as $a\to 0$.

 }
\end{eap}

\begin{pp} \label{r0m}
  $\rho_1$ is stronger than $\rho_0$ on $F^1_{USC} (X)$.
\end{pp}

\begin{proof}
  The desired result follows immediately from Lemma \ref{der}.
\end{proof}

The following statement may be a known result. But we can not find the original literature which presents this fact.

\begin{lm} \label{lrm}
 Let $\{u_n\} \subset    K(X)$ satisfy
$ u_1 \supseteq u_2 \supseteq \ldots \supseteq u_n \supseteq \ldots.  $
Then
$u= \bigcap_{n=1}^{+\infty}  u_n \in K(X)$ and
$H(u_n, u) \to 0  \ \hbox{as} \ n\to \infty$.

On the other hand, if $\{u_n\} \subset    K(X)$, $ u_1 \subseteq u_2 \subseteq \ldots \subseteq u_n \subseteq \ldots  $
and
$u= \overline{\bigcup_{n=1}^{+\infty}  u_n }\in K(X)$,
then
$H(u_n, u) \to 0  \ \hbox{as} \ n\to \infty$.

\end{lm}

\begin{proof}
  The desired result follows from the basic topology.
\end{proof}

Suppose that $u\in F^1_{USCG} (X)$, $\al, \beta \in [0,1]$ and $\al <  \beta$.
The ``variation"   $w_u(\al, \beta)$ is defined as
$$w_u(\al, \beta): = \sup \{ H([u]_\xi,  [u]_\eta):  \xi,\eta \in (\al, \beta]  \}.$$

The concept ``$w_u(\al, \beta)$" is from \cite{joo}. The following Lemma \ref{vem} is the version of fuzzy sets on metric space $(X,d)$ of Lemma 3.2 in \cite{joo}.

\begin{lm} \label{vem}
Suppose that $u \in F^1_{USCB} (X)$. Given $\varepsilon>0$. Then there exist
$\al_1$, $\al_2$, \ldots, $\al_k$ such that
$0=\al_1< \al_2<\cdots<\al_k=1$
and
$$w_u(\al_{i}, \al_{i+1} ) < \varepsilon, \   i=1,2,\ldots, k-1.$$

\end{lm}

\begin{proof}
    Consider
$[u](\cdot) : [0,1] \to (K(X), H)$, which is the cut-function
of
$u$ and      is defined by $[u](\al) = [u]_\al$.

From Lemma \ref{lrm},
$[u](\cdot)$
is
 left-continuous at $\al \in  (0,1]$
and
$$\lim_{\zeta \to h+} H([u]_\zeta, \overline{\bigcup \{[u]_\al : \al > h \}} ) =0 $$
for
$h\in [0,1)$.

The remainder proof can proceed similarly to the proof of Lemma 3.2 in \cite{joo}.
\end{proof}

\begin{pp} \label{esp}
  $\rho_0$ and $\rho_1$ are equivalent on $F^1_{USCB} (X)$.
\end{pp}

\begin{proof}
Using Lemma \ref{vem},
  the proof can proceed similarly to the proof of Lemma 3.7 and Theorem 3.8 in \cite{joo}.
\end{proof}

However, on
$F^1_{USCG}(X)^p$, which is larger than $F^1_{USCB}(X)$,
$\rho_0$ and $\rho_1$ need not be equivalent. A counterexample is given in the following.

\begin{eap}   \label{r0n}{\rm
    Consider     $u\in F^1_{USCG}(\mathbb{R})$ defined as
$$[u]_\gamma = [0, \gamma^{-0.4}]$$ for all $\gamma>0$.

 Clearly, $\rho_{0} (t_a u, u) \to 0$ as $a\to 0$, where $t_a$ is defined as in Example \ref{tgf}.
However, it can be checked that $\rho_1 (t_a u, u) \not\to 0$ as $a\to 0$.
In fact, note that
$S(t) \geq \ln \frac {\beta}{t^{-1}\beta}$ when $\beta\in (0,1]$,
so
 for each $t\in T$ with $S(t)< +\infty$, there is a $K \geq 1$ such that
$t^{-1}(\beta) \geq  \beta/K$ for all $\beta \in [0,1]$.
Thus for $a\in (0,1)$
\begin{align*}
 d_\infty & (tu, t_a u)
\\
& \geq  \sup_{\beta < a}   H   (  [tu]_{ \sqrt{\beta}   },      [t_a u]_{ \sqrt{\beta}  }    )
\\
&= \sup_{\beta < a} H(  [u]_{ t^{-1} \sqrt{\beta}   },      [ u]_{ \beta  }    )  =  +\infty,
\end{align*}
and then
$\rho_1(u, t_a u) = +\infty$.

Note that $u$
and $t_a u $ belong to $F^1_{USCG} (\mathbb{R})^1$,
so
 $\rho_0$ convergence need not imply $\rho_1$ convergence on $F^1_{USCG}(\mathbb{R})^1$.
A similar example can show that
 the $\rho_0$ convergence need not imply the $\rho_1$ convergence on $F^1_{USCG}(\mathbb{R})^p$.

}
\end{eap}

Clearly, Example \ref{r0n} indicates
that
$D(t) \to 0$ need not imply $S(t) \to 0$.
The analysis in Example \ref{r0n} also implies that $S(t_a) = +\infty$ for $a\in (0,1)$.

 Wu, Zhang and Chen \cite{wug}
 proposed an example that a contraction whose Zadeh's extension is not a contraction under the Skorokhod metric and negatively answered
 the corresponding
 questions asked by Jard\'{o}n, S\'{a}nchez and Sanchis \cite{jarn}.
In \cite{huang10}, we give a simple example to answer the questions.

\section{Relation between Skorokhod metric $\rho_0$ and $d_p$ metric on $F^1_{USCB} (X)$}
\label{uscb}

In this section, it is shown that the Skorokhod metric $\rho_0$ is stronger than the $d_p$ metric on $F^1_{USCB} (X)$.
However,
this is not the case with $F^1_{USCG} (X)^p$, which is larger than $F^1_{USCB} (X)$.
A counterexample
is given
to show that
 the Skorokhod metric $\rho_0$ convergence need not imply the $d_p$ metric convergence
 on
$F^1_{USCG} (\mathbb{R})^p$.

\begin{lm} \label{pen} Let $u \in F_{USCB}^1(X)$.
Then
  $d_p(u, tu) \to 0$ as $D(t) \to 0$.
\end{lm}

\begin{proof}
  Given $\varepsilon>0$. From Lemma \ref{vem},
there exist points $\al_1$, \ldots, $\al_k$ such that
$0=\al_1 < \al_2 < \al_k=1$
and
$w_u(\al_l, \al_{l+1})  \leq \varepsilon/3$ for all $1\leq l \leq k-1$.

Let
$$ M := H([u]_0, [u]_1).  $$
Then
for each $t \in T$,
\begin{align*}
  d_p   (u, t u) &   =   \left(\int_0^{1}    H([u]_\alpha, [t u]_\al )^p \;d\alpha  \right)^{1/p}
 \\
 & \leq  \sum_{l=1}^{k-1} \left(\int_{\al_{l}}^{\al_{l+1}}     H([u]_\alpha, [t u]_\al )^p \;d\alpha  \right)^{1/p}
\\
& \leq (k-1)  M \cdot (2 D(t) )^{1/p}     +   \varepsilon/3.
\end{align*}
Thus, there is a $\zeta(\varepsilon)$ such that
 \begin{equation*} d_p   (u, t u)   \leq   \varepsilon \end{equation*}
for all $t\in T$ with $D(t) < \zeta$.

\end{proof}

\begin{tm} \label{dpc}
Suppose that $u \in F^1_{USCB}(X)$, $u_n \in F^1_{USC}(X)$, $n=1,2,\ldots$.
If $\rho_0(u_n,u) \to 0$, then $d_p(u_n, u) \to 0$.

\end{tm}

\begin{proof}
Given $\varepsilon>0$. From Lemma \ref{pen} there is a $\zeta>0$ such that
$d_p(u, tu) < \varepsilon/2$
for all $D(t) < \zeta$.

Since
 $\rho_0(u_n, u) \to 0$, then there exists an $N$
such that
 $\rho_0(u_n, u) < \eta= \min\{\varepsilon/2,  \zeta\}$ for $n \geq N$.
This means
that, for each $n \geq N$,
there is a $t_n \in T$ such that
$D(t_n) < \eta $
and
$d_\infty (u_n, t_n u) < \eta$.

So
\begin{align*}
  d_p  (u_n, u) & \leq   d_p (u_n, t_n u) +   d_p (t_n u,  u)
\\
& \leq   \eta + \varepsilon/2 \leq \varepsilon
\end{align*}
for all $n \geq N$.

\end{proof}

The converse of the implication in Theorem \ref{dpc} does not hold.
$\{u_n\}$ and $u$ in
Example \ref{enr} is a counterexample
 shows that
the $d_p$ metric convergence
need not imply
the Skorokhod metric $\rho_0$ convergence on $F^1_{USCB} (\mathbb{R})$.

 Theorem \ref{dpc} is not true
if
$F^1_{USCB} (X) $ is replaced
by
$F^1_{USCG} (X)^p $, which is larger than $F^1_{USCB} (X) $.
An
example
is given in the following
to show
that
the Skorokhod metric $\rho_0$ convergence need not imply
the $d_p$ metric convergence
on
  $F^1_{USCG} (\mathbb{R})^p $.

\begin{eap} \label{r0np}
{\rm

 Consider $t_{a,\theta}$, $(a,\theta) \in (0,1) \times  (0,1)$, defined as
 \[
t_{a,\theta} (\xi)=\left\{
      \begin{array}{ll}
        \xi^\theta, & \xi\in [0,a], \\
         \dfrac{  1-a^\theta     }    {1-a} \xi +  \dfrac{   a^\theta - a  }  {1-a}, &   \xi\in [a,1].
      \end{array}
    \right.
\]
Then
$t_{a,\theta} \in T$.

 Consider
 $u\in F^1_{USCG}(\mathbb{R})$ defined as
$$[u]_\gamma = [0, \gamma^{-0.6}]$$ for all $\gamma>0$.
Then
\begin{align*}
 d_1(u, \widehat{\{0\}}) & =   \int_0^{1}  H([u]_{\alpha}, \{0\})  \;d\alpha
\\
& = \int_0^{1} \alpha^{-0.6}  \;d\alpha  =2.5,
\end{align*}
and therefore $u \in  F^1_{USCG}(\mathbb{R})^1 $.

Note that
\begin{align*}
 d_1((t_{a, \theta} u, \widehat{\{0\}}) & =   \int_0^{1}  H([t_{a, \theta} u]_{\gamma}, \{0\})  \;d\gamma
\\
& = \int_0^{a^\theta} \gamma^{-0.6/\theta}  \;d\gamma    +    \int_{a^\theta}^1    H([t_{a, \theta} u]_{\gamma}, \{0\})  \;d\gamma,
\end{align*}
and
\begin{align*}
  d_1 (t_{a, \theta}u, u) & = \int_0^{1}  H([t_{a,\theta} u]_{\gamma}, [u]_\gamma) \;d\gamma
\\
& \geq    \int_0^{a^\theta}  H([t_{a,\theta} u]_{\gamma}, [u]_\gamma) \;d\gamma
\\
&  =   \int_0^{a^\theta} |\gamma^{-0.6/\theta}  -   \gamma^{-0.6} | \;d\gamma,
\end{align*}
thus there exist $a_n \to 0+$ and $\theta_n \to 0.6+$ such that
\begin{gather*}
  d_1((t_{a_n, \theta_n} u, \widehat{\{0\}})     <   +\infty,
 \\
          d_1 (t_{a_n, \theta_n}u, u) > 0.5.
 \end{gather*}
 So
$t_{a_n, \theta_n} u \in F^1_{USCG}(\mathbb{R})^1 $,
$\rho_0( t_{a_n, \theta_n} u, u   ) \to 0$ and  $ d_1 (t_{a_n, \theta_n}u, u) \not\to 0 $.

It can be shown by a similar example
 that the $\rho_0$ convergence
need not imply
the $d_p$ convergence on $ F^1_{USCG}(\mathbb{R})^p$.

}
\end{eap}

\section{Relation between Skorokhod-type metric $\rho_1$ and $d_p$ metric on $F^1_{USCG}(X)^p$} \label{uscglp}

In this section,
 we
first discuss some basic properties of $u$ in $F^1_{USCG}(X)^p$.
Then
we find a fact
that for $u \in F^1_{USCG}(X)^p$,
$d_p(u,tu)\to 0$ as $S(t)\to 0$.
Based on this,
we
show
that the Skorokhod-type metric $\rho_1$ is stronger
than
 the $d_p$ metric on $F^1_{USCG}(X)^p$.
A counterexample
is given to show that the $d_p$ metric is not stronger than
the Skorokhod-type metric $\rho_1$
on
$F^1_{USCG}(\mathbb{R})^p \setminus F^1_{USCB}(\mathbb{R})$.

\begin{tm} \label{tulg}
  Suppose that $u \in F^1_{USCG}(X)^p$ and $t\in T$. If $S(t)< +\infty$,
then
  $tu \in    F^1_{USCG}(X)^p$.
\end{tm}

\begin{proof} \label{tup}
Suppose that
$S(t) < +\infty$.
Then there is a $K \geq 1$ such
that
$t^{-1}(\beta) \geq  \beta/K$ for all $\beta \in [0,1]$.

Thus
\begin{align*}
  d_p     (t u, \widehat{\{x_0\}}  )
   & =  \left(\int_0^{1}    H([tu]_\alpha, \{x_0\})^p \;d\alpha  \right)^{1/p}
\\
&  \leq   \left(\int_0^{1}    H([u]_{\alpha/K}, \{x_0\})^p \;d\alpha  \right)^{1/p}
\\
&    =    \left(\int_0^{1/K}   K H([u]_{\alpha}, \{x_0\})^p \;d\alpha  \right)^{1/p}
\\
&   \leq  K^{1/p}    \left(\int_0^{1}  H([u]_{\alpha}, \{x_0\})^p \;d\alpha  \right)^{1/p}
\\
&  =     K^{1/p}  d_p     (u, \widehat{\{x_0\}} ) .
\end{align*}
So  $tu \in    F^1_{USCG}(X)^p$.

\end{proof}

\begin{tm}
Suppose that $t\in T$ with $S(t)< +\infty$.
Then
    $u \in F^1_{USCG}(X)^p$
    is equivalent to
  $tu \in    F^1_{USCG}(X)^p$.
\end{tm}

\begin{proof}
 Note that $u=t^{-1} t u$, thus the desired result follows from Theorem \ref{tulg}.
\end{proof}

If the condition $S(t) < +\infty$ is reduced to the condition $D(t) < +\infty$,
then
the conclusion in
Theorem \ref{tulg} does not hold.
A counterexample
is
given in the following.

\begin{eap}\label{dnp}{\rm
  Consider
$t_a \in T$, $a\in (0,1)$,   in Example \ref{tgf}
and
 $u$ given in Example \ref{r0np}.
Then
\begin{align*}
 d_1(u, \widehat{\{0\}}) & =   \int_0^{1}  H([u]_{\alpha}, \{0\})  \;d\alpha
\\
& = \int_0^{1} \alpha^{-0.6}  \;d\alpha  =2.5,
\end{align*}
and therefore $u \in  F^1_{USCG}(\mathbb{R})^1 $.

Then
\begin{align*}
  d_1 (t_{0.3}u, \widehat{\{0\}}) & = \int_0^{1}  H([t_{0.3} u]_{\alpha}, \{0\}) \;d\alpha
\\
& \geq    \int_0^{\sqrt{0.3}}  H([t_{0.3} u]_{\alpha}, \{0\}) \;d\alpha
\\
&  =   \int_0^{\sqrt{0.3}}  \alpha^{-1.2}  \;d\alpha  =  +\infty.
\end{align*}
So
$t_{0.3}u \notin    F^1_{USCG}(\mathbb{R})^1$. In fact, it can be checked
that
$t_{a}u \notin    F^1_{USCG}(\mathbb{R})^1$ for $a\in (0,1)$.

Note that $D(t_a) \to 0$ as $a\to 0$.
So
even if $u \in F^1_{USCG}(X)^p$ and $t\in T$ with $D(t) $ being less than any positive number required,
$tu$ is still not   necessarily in
$F^1_{USCG}(X)^p$.

}
\end{eap}

\begin{re}
{\rm
  Example \ref{r0np} indicates that there exist $u\in F(\mathbb{R})$ and $t\in T$ such that
$S(t)=+\infty$ and $u$ and $tu$ are both in $F^1_{USCG}(\mathbb{R})^p$.
}
\end{re}

To show $\rho_1$ convergence can imply $d_p$ convergence on $F^1_{USCG} (X)^p$,
we
need
a fact
that
 $d_p(u, tu) \to 0$ as $S(t) \to 0$ when $u\in F^1_{USCG} (X)^p$. We begin with some lemmas.

\begin{lm} \label{vemg}
Suppose that $u \in F^1_{USCG} (X)$. Given $h>0$ and $\varepsilon>0$. Then there exist
$\al_1$, $\al_2$, \ldots, $\al_k$ and $\delta>0$ such that
$h=\al_1< \al_2<\cdots<\al_k=1$,
\begin{gather*}
  w_u(h, h-\delta ) < \varepsilon,  \   \mbox{and}
\\
  w_u(\al_{i}, \al_{i+1} ) < \varepsilon, \   i=1,2,\ldots, k-1.
\end{gather*}

\end{lm}

\begin{proof}
Note that the cut-function $[u](\cdot) : [0,1] \to (C(X), H)$
is
 left-continuous at $h$,
the
proof can proceed similarly to that of Lemma \ref{vem}.

\end{proof}

\begin{lm} \label{edpg}
  Suppose that     $u \in F^1_{USCG}(X)^p$ and $h>0$.
Then
$$ \left(\int_h^{1}    H([u]_\alpha, [t u]_\al )^p \;d\alpha  \right)^{1/p}  \to 0$$
as
 $D(t) \to 0$.
\end{lm}

\begin{proof}
By using Lemma \ref{vemg}, the proof can proceed similarly
 to the proof of
  Lemma \ref{pen}.

\end{proof}

The following important property of Lebesgue integral is useful in the proof of Theorem \ref{dpu}.

\begin{itemize}
  \item
\textbf{Absolute continuity of Lebesgue integral}. Suppose that $f$ is Lebesgue integrable on $E$, then for arbitrary $\varepsilon>0$, there is a $\delta>0$ such that $\int_A f \;dx < \varepsilon$ whenever $A \subseteq E$ and $m(A)<\delta$.

\end{itemize}

\begin{tm} \label{dpu}

Let $u \in F^1_{USCG}(X)^p$.
Then
  $d_p(u, tu) \to 0$ as $S(t) \to 0$.
 \end{tm}

\begin{proof}

Given $\varepsilon>0$.
From the absolute continuity of Lebesgue integral,
there
is a $\theta>0$ such that
for all $0\leq h \leq   \theta$
\begin{gather*}\label{h1e}
  \left(\int_0^{h}    H([u]_\alpha, \{x_0\})^p \;d\alpha  \right)^{1/p}   \leq   \varepsilon/3.
\end{gather*}

Choose $\xi>0$ satisfies that if $S(t) < \xi$
then
$t^{-1} (\alpha) > \al / 1.1$ for all $\al \in [0,1]$.
Thus, for all $0\leq h \leq  \theta$ and $t\in T$ with $S(t) < \xi$
\begin{align*}
 \mbox{} & \mbox{} \left(   \int_0^{h}  H  ([tu]_\alpha, \{x_0\})^p \;d\alpha  \right)^{1/p}
\\ &=
 \left(\int_0^{h}    H([u]_{t^{-1}(\alpha)}, \{x_0\})^p \;d\alpha  \right)^{1/p}
\\
& \leq  \left(\int_0^{h}    H([u]_{\alpha/1.1}, \{x_0\})^p \;d\alpha  \right)^{1/p}
\\
& =    \left(\int_0^{h/1.1}   1.1 H([u]_{\alpha}, \{x_0\})^p \;d\alpha  \right)^{1/p}
\\
& \leq 1.1 \varepsilon/3.
\end{align*}

From Lemmas    \ref{der} and \ref{edpg},
there is a
$\eta>0$ such that $$ \left(\int_\theta^{1}    H([u]_\alpha, [t u]_\al )^p \;d\alpha  \right)^{1/p} < \varepsilon/6$$
when
$S(t) < \eta$.

So for $t\in T$ with $S(t) < \zeta= \min\{\xi, \eta\}$
\begin{align*}
  d_p  & (u, t u) =   \left(\int_0^{1}    H([u]_\alpha, [t u]_\al )^p \;d\alpha  \right)^{1/p}
 \\
 & \leq  \left(\int_0^{\theta}    H([u]_\alpha, [t u]_\al )^p \;d\alpha  \right)^{1/p}
+
\left(\int_\theta^1    H([u]_\alpha, [t u]_\al )^p \;d\alpha  \right)^{1/p}
 \\
   & \leq    \left(\int_0^{\theta}    H([u]_\alpha, \{x_0\})^p \;d\alpha  \right)^{1/p}
+
 \left(\int_0^{\theta}    H([t u]_\alpha, \{x_0\})^p \;d\alpha  \right)^{1/p}
+
   \left(\int_\theta^1    H([u]_\alpha, [t u]_\al )^p \;d\alpha  \right)^{1/p}
\\
& \leq \varepsilon/3 +  1.1\varepsilon/3 +  \varepsilon/6  < \varepsilon.
\end{align*}

\end{proof}

\begin{tm} \label{prs}
  Suppose that $u\in  F^1_{USCG}(X)^p$ and $u_n \in F^1_{USC} (X)$, $n=1,2,\ldots$.
If
 $\rho_1(u_n, u) \to 0$, then $d_p(u_n, u) \to 0$.
\end{tm}

\begin{proof}
The proof is similarly to that of Theorem \ref{dpc}.

Given $\varepsilon>0$. From Theorem \ref{dpu} there is a $\zeta>0$ such that
$d_p(u, tu) < \varepsilon/2$
for all $S(t) < \zeta$.

Since
 $\rho_1(u_n, u) \to 0$, then there exists an $N$
such that
 $\rho_1(u_n, u) < \nu= \min\{\varepsilon/2,  \zeta\}$ for $n \geq N$.
This means
that, for each $n \geq N$,
there is a $t_n \in T$ such that
$S(t_n) < \nu $
and
$d_\infty (u_n, t_n u) < \nu$ .

So
\begin{align*}
  d_p  (u_n, u) & \leq   d_p (u_n, t_n u) +   d_p (t_n u,  u)
\\
& \leq   \nu + \varepsilon/2 \leq \varepsilon
\end{align*}
for all $n \geq N$.

\end{proof}

Based on the results obtained in this paper, it can be seen 
that
the conclusion in Theorems \ref{dpu} and \ref{prs} can also be proved by using the
the Lebesgue's Dominated Convergence Theorem.

In this section, we know that
the Skorokhod-type metric $\rho_1$ is stronger than the $d_p$ metric on $F^1_{USCG}(X)^p$. In Section \ref{uscb}, we find
 that
the Skorokhod metric $\rho_0$ is not necessarily
stronger
than the $d_p$ metric on $F^1_{USCG}(X)^p$.
 These facts
indicate that
the Skorokhod metric $\rho_0$ is not necessarily equivalent to the  Skorokhod-type metric $\rho_1$
on
$F^1_{USCG}(X)^p$.
This is a conclusion in Section 3.

The following example
is given to show that the $d_p$ metric is not stronger than
the Skorokhod-type metric $\rho_1$
on
$F^1_{USCG}(\mathbb{R})^p \backslash F^1_{USCB}(\mathbb{R})$.

\begin{eap}
{\rm
  Consider
  $u\in   F^1_{USCG}(\mathbb{R})^1 \setminus F^1_{USCB}(\mathbb{R})$ defined as
$$[u]_\gamma = [0, \gamma^{-0.6}]$$ for all $\gamma>0$,
and
$t_{0.3,\theta}$, $\theta \in  (0,1)$ defined in Example \ref{r0np}.
Then $t_{0.3, \theta} u \in   F^1_{USCG}(\mathbb{R})^1 \setminus  F^1_{USCB}(\mathbb{R})$ for $\theta \in (0.6,1)$.

It
can be checked
that
$\rho_1(u, t_{0.3, \theta_n} u ) \not\to 0$ as $\theta_n \to 1-$.

On the other hand,
it follows from the Lebesgue's Dominated Convergence Theorem
that
$d_1(u, t_{0.3, \theta_n} u ) \to 0$ as $\theta_n \to 1-$.

So the $d_1$ metric is not stronger than the Skorokhod-type metric $\rho_1$ on  $F^1_{USCG}(\mathbb{R})^1 \setminus F^1_{USCB}(\mathbb{R})$.
A similarly example can
show
that
 the $d_p$ metric is not stronger than the Skorokhod-type metric $\rho_1$ on  $F^1_{USCG}(\mathbb{R})^p \setminus F^1_{USCB}(\mathbb{R})$.

}
\end{eap}

\begin{re}{\rm
  From this example, we can see that even if a fuzzy set sequence is both $\rho_0$ convergence and $d_p$ convergence, it is not necessarily be $\rho_1$ convergence. 
  
  Let $v$ be defined as $[v]_\gamma = [0, \gamma^{-0.4/p}]$, $\gamma >0$, and let $t_{1/n}$ be defined as in Example \ref{tgf}.
  Similarly, it can be checked that $\{t_{1/n} v\}$ is a sequence in  $F^1_{USCG}(\mathbb{R})^p \setminus F^1_{USCB}(\mathbb{R})$ which is  $\rho_0$ convergence and $d_p$ convergence but is not $\rho_1$ convergence.
}
\end{re}

\section{Relation between Skorokhod-type metric $\rho_1$ and $d_p$ metric on $F^1_{USC}(X)$} \label{usc}

In this section,
we show
that, unlike the case of $F^1_{USCG}(X)^p$,
the $\rho_1$ convergence    is not necessarily the $d_p$ convergence on $F^1_{USC}(X)$ by a counterexample.

The following example indicates
that
the
 $\rho_1$ convergence    is not necessarily the $d_p$ convergence on $F^1_{USCG}(X) \backslash  F^1_{USCG}(X)^p$.

\begin{eap}
  {\rm
    Consider    $u\in  F^1_{USCG}(\mathbb{R}) \backslash  F^1_{USCG}(\mathbb{R})^p$ defined as
$$[u]_\gamma = [0, \gamma^{-1.6}]$$ for all $\gamma>0$.

Let $t^a$, $a\in (0,1)$, defined as
\[
 t^a(\xi)=\left\{
            \begin{array}{ll}
              (1+a)\xi,    & \xi\in [0,\frac{1}{2}], \\
               (1-a) \xi + a,   &  \xi \in [\frac{1}{2}, 1].
            \end{array}
          \right.
\]
Then $t^a\in T$,
$t^a u \in  F^1_{USCG}(\mathbb{R}) \backslash  F^1_{USCG}(\mathbb{R})^p$
 for $a\in (0,1)$, and
 $\rho_1 (u, t^a u) \to 0$ as $a\to 0$.

On the other hand, for $a\in (0,1)$
  \begin{align*}
  d_p (u, t^a u) & =  \left(  \int_0^1   H([u]_\gamma, [t^a u]_\gamma)^p \, d\gamma \right)^{1/p}
  \\
  & \geq  \left(  \int_0^{1/2}   H([0, \gamma^{-1.6}],  \ \   [0,    (  (1+a)^{-1} \gamma  )^{-1.6} ]  )^p \, d\gamma \right)^{1/p}
  \\
  & = +\infty.
  \end{align*}

So the $\rho_1$ convergence need not imply the $d_p$ convergence on $F^1_{USCG}(\mathbb{R}) \backslash  F^1_{USCG}(\mathbb{R})^p$.

  }
\end{eap}

It can also be checked
that
the
 $\rho_1$ convergence    is not necessarily the $d_p$ convergence on $F^1_{USC}(X) \backslash  F^1_{USCG}(X)$.

From the results in Sections \ref{uscb}, \ref{uscglp} and \ref{usc},
the following statements are true
for $u$, $u_n$ in $F^1_{USC}(X)$, $n=1,2,\ldots$.     .
\begin{enumerate}
\renewcommand{\labelenumi}{(\roman{enumi})}

 \item  If $u \in F^1_{USCB}(X)$, then $\rho_0(u_n, u) \to 0$ can imply $d_p(u_n, u) \to 0$.

\item   If $u \in F^1_{USCG}(X)^p$, then $\rho_1(u_n, u) \to 0$ can imply $d_p(u_n, u) \to 0$. However
 $\rho_0(u_n, u) \to 0$
need not imply $d_p(u_n, u) \to 0$.

 \item   If $u \in F^1_{USC}(X)$, then
 $\rho_1(u_n, u) \to 0$
need not imply $d_p(u_n, u) \to 0$.

\end{enumerate}

\section{Relation between Skorokhod metric and sendograph metric on $F_{USC}^1(X)$}

In this section,
it is found that
 the Skorokhod metric is stronger than the sendograph metric and the endograph metric on $F_{USC}^1(X)$.
The sendograph metric is stronger than the endograph metric (see Section \ref{fsm}).
A
counterexample is given
to show
that the sendograph metric convergence need not imply the Skorokhod metric convergence on $F_{USCB}^1(\mathbb{R})$.

For $u\in F_{USC}^1(X)$, the symbol $P_0(u)$ is used to denote the set $\{\al\in (0,1):  \lim_{\beta\to \al} H([u]_\beta, [u]_\al) \not=0 \}$.

\begin{tm} \label{d0e}
  Let
 $u_n$, $u$, $n=1,2,\ldots$, be fuzzy sets in $F_{USC}^1(X)$.
If $\rho_0(u_n,  u) \to 0$,
  then
\\
(\romannumeral1)  $H([u_n]_0, [u]_0) \to 0$,
\\
(\romannumeral2)  $H([u_n]_1, [u]_1) \to 0$,
\\
(\romannumeral3)  $H_{\rm end} (u_n, u) \to 0$,
\\
(\romannumeral4)  $H_{\rm send} (u_n, u) \to 0$,
and
\\
(\romannumeral5)  $H ([u_n]_\al, [u]_\al) \to 0$ for all $\al\in (0,1) \backslash P_0(u)$.
\end{tm}

\begin{proof}

Note that $t(0)=0$ and $t(1)=1$ for each $t\in T$.
So
\begin{gather*}
    \rho_0(u,v) \geq  H([u]_0, [v]_0),
     \\
   \rho_0(u,v) \geq  H([u]_1, [v]_1)
  \end{gather*}
for all $u,v \in F_{USC}^1(X)$ and
therefore
(\romannumeral1) and (\romannumeral2) are true.

To prove (\romannumeral3).
Given $\varepsilon>0$.
Since
$\rho_0(u_n, u) \to 0$, then there exists $N$, for each $n \geq N$, there is a $t_n \in T$ such that
$d_\infty (t_n u_n, u) < \varepsilon/2$
and
$D(t_n) < \varepsilon/2$.
Thus
\begin{align*}
H^* & ({\rm end}\,u, {\rm end}\, u_n)
\\
& = \sup \{  d  ( (x,\al),  {\rm end}\,u_n ) : (x,\al) \in {\rm end}\,u \}
\\
   & \leq  \sup \{   H([u]_\al,   [u_n]_{t_n^{-1}(\al)} )    +   \varepsilon/2   : (x,\al) \in {\rm end}\,u \}
\\
&\leq d_\infty (t_n u_n, u)   + \varepsilon/2
\\
& \leq \varepsilon
\end{align*}
and
\begin{align*}
H^* & ({\rm end}\,u_n, {\rm end}\, u)
\\
& = \sup \{  d  ( (x,\al),  {\rm end}\,u ) : (x,\al) \in {\rm end}\,u_n \}
\\
   & \leq  \sup \{   H([u_n]_\al,   [u]_{t_n(\al)} )    +   \varepsilon/2   : (x,\al) \in {\rm end}\,u_n \}
\\
&\leq d_\infty (t_n u_n, u)   + \varepsilon/2
\\
& \leq \varepsilon.
\end{align*}
From the arbitrariness of $\varepsilon>0$,
$$
H_{\rm end} (u_n, u)
=
 \max\{ H^*  ({\rm end}\,u, {\rm end}\, u_n), H^*  ({\rm end}\,u_n, {\rm end}\, u)\} \to 0.$$
So (\romannumeral3) is true.

The proof of (\romannumeral4) is very similar to that of  (\romannumeral3).

To prove (\romannumeral5),
  suppose that $\al \in (0,1)\backslash P_0(u)$.
  Given $\varepsilon>0$.
  There exists a $\delta>0$
  such that
  \begin{equation}\label{usf}
  H([u]_\beta, [u]_\al) < \varepsilon/2
  \end{equation}
   for all $\beta \in
(\al-\delta, \al+\delta)$.

  From $\rho_0(u_n,  u) \to 0$, we know
  that
  there is an $N$ such that $\rho_0(u_n,  u) < \zeta=\min \{\delta, \varepsilon/2\}$ for all $n \geq N$.
This means that for each $n \geq N$,
 there is
 a $t_n$ such that
 \begin{equation}\label{ulg}
   d_\infty ( u_n, t_n u) < \zeta
  \   \mbox{and}  \
  D(t_n)< \zeta
 \end{equation}
By \eqref{usf} and \eqref{ulg}, for all $n \geq N$,
  \begin{align*}
                    H  & ([u_n]_\al, [u]_\al)
                     \\
                     & \leq    H([u_n]_\al, [u]_{t_n^{-1} (\al) } )   +      H([u]_{t_n^{-1} (\al) }, [u]_\al)
                     \\
                      & \leq    d_\infty ( u_n, t_n u)   +      H([u]_{t_n^{-1} (\al) }, [u]_\al)
                     \\
                      & <  \zeta + \varepsilon/2    \leq \varepsilon.
  \end{align*}
From the arbitrariness of $\varepsilon>0$, $ H [u_n]_\al, [u]_\al) \to 0$.
So
(\romannumeral5) is true.

  \end{proof}

\begin{re}{\rm
  The (\romannumeral4) in Theorem \ref{d0e} can also be deduced from the relation of sendograph metric and endograph metric (see Section \ref{fsm}) and
the (\romannumeral1) and (\romannumeral3)  in Theorem \ref{d0e}.
}
\end{re}

Theorem \ref{d0e} indicates that the Skorokhod metric convergence
can
imply
the sendograph  metric convergence on $F_{USC}^1(X)$.
However, the converse implication does not hold.
The
following is an example
of a sequence in $F_{USCB}^1(\mathbb{R})$
which is
sendograph metric convergence but is not Skorokhod metric convergence.

\begin{eap} \label{enr}{\rm
  Consider
\[
     u(x)=\left\{
           \begin{array}{ll}
             1, & x=0, \\
             \frac{1}{2}, & x\in (0,2], \\
              0, & x\notin [0,2],
           \end{array}
         \right.
\]
and
\[
  u_n(x)=\left\{
           \begin{array}{ll}
             1   -  \frac{1}{2}  x^{1/n}, & x\in [0,1], \\
             \frac{1}{2} (1-(x-1)^n), & x\in [1,2], \\
             0, & x\notin [0,2],
           \end{array}
         \right.  n=1,2,\ldots.
    \]
So $u$ and $u_n$, $n=1,2,\ldots$ are in $F^1_{USCB}(\mathbb{R})$, and
\[
     [u]_\al=\left\{
           \begin{array}{ll}
             \{0 \}, & \al\in (1/2, 1], \\
             \mbox{} [0,  2 ], &  \al \in [0, 1/2],
           \end{array}
         \right.
\]
and
\[
     [u_n]_\al=\left\{
           \begin{array}{ll}
             [0, (2-2\alpha)^n], & \al\in [1/2, 1], \\
             \mbox{} [0,   1+ (1-2\alpha)^{1/n} ], &  \al \in [0, 1/2].
           \end{array}
         \right.
\]
Note that $[u_n]_{1/2}  \equiv   [0,1]$, so for all $n=1,2,\ldots$,
$$\rho_0(u_n,u) \geq 1 .$$
In fact it can be
checked
that
 $\rho_0(u_n,u) \equiv 1$.

On the other hand, since
$$H([u_n]_\al, [u]_\al) \to 0 \ \mbox{for all}\ \al\in [0,1]\backslash \{ \frac{1}{2}\}.$$
Thus by
Theorem 6.4 in \cite{huang},
$H_{\rm end} (u_n, u) \to 0$, and then $H_{\rm send} (u_n, u) \to 0$.
So
$\{u_n\}$ and $u$ satisfy
statements
(\romannumeral1)-(\romannumeral5) in Theorem \ref{d0e}.
But
$\rho_0(u_n, u) \not\to 0$.

In addition, we can see that $d_p (u_n, u) \to 0$.
So this example also
indicates
that
the $d_p$ metric convergence
need not imply
the Skorokhod metric $\rho_0$ convergence
on $F^1_{USCB} (\mathbb{R})$.
This fact can also be derived from the conclusions in \cite{huang}, see Section \ref{cou}.

}
\end{eap}

\section{Conclusion} \label{cou}

In this paper, we first discuss the relation between the Skorokhod metric $\rho_0$ and   the Skorokhod-type metric $\rho_1$.
$\rho_1$ is stronger than $\rho_0$ on $F^1_{USC} (X)$.
It is found
that
$\rho_1$ is equivalent to $\rho_0$ on $F^1_{USCB} (X)$,
and
that
$\rho_1$ is not necessarily equivalent to $\rho_0$ on $F^1_{USCG} (X)^p$.

Then
we investigate
 relation between these two metrics and $d_p$ metric.
It is found
that
the compactness of $\al$-cuts and the integrability of fuzzy sets play important roles.
On $F^1_{USCB} (X)$,
the Skorokhod metric $\rho_0$  is stronger than the $d_p$ metric.
 On $F^1_{USCG} (X)^p$,
the Skorokhod-type metric $\rho_1$ is still stronger than the $d_p$ metric,
however the
Skorokhod metric $\rho_0$  is not necessarily stronger than the $d_p$ metric.
On $F^1_{USC} (X)$,
even the
Skorokhod-type metric $\rho_1$  is not necessarily stronger than the $d_p$ metric.
We also show that the Skorokhod metric $\rho_0$  is stronger than the sendograh metric on $F^1_{USC} (X)$.

Our recent results on level decomposition properties
of
the endograph metric can immediately imply that
 $H_{\rm send}(u_n, u) \to 0$ is equivalent to $d_p(u_n, u) \to 0$ and $H ([u_n]_0, [u]_0) \to 0$ on $F^1_{USCB} (\mathbb{R}^m)$
(see the end of Section 6 or Theorem 6.4 in \cite{huang}).
So
the statement in \cite{kim}
that
the Skorokhod metric $\rho_0$ is stronger than the sendograph metric
on $F^1_{USCB} (\mathbb{R}^m)$
can be derived
from
the statement in \cite{joo2} that the Skorokhod metric $\rho_0$ is stronger than the $d_p$ metric
on $F^1_{USCB} (\mathbb{R}^m)$.
We
find some interesting
relations among the metrics on fuzzy sets which will be presented in the future work.

\appendix
\section{  The proof of Theorem \ref{dpsm}   }

\begin{proof}
  To prove that $d_p^*$ is a metric, we need to show that, for all $u,v, w$ in $F^1_{USC} (X)$,
\\
(\romannumeral1) \
$d_p^*(u,v)\geq 0$ and
$d_p^*(u,v)=0$ is equivalent to $u=v$,
\\
(\romannumeral2) \ $d_p^*(u,v) = d_p^*(v,u)$, and
\\
(\romannumeral3) \
$d_p^*(u,v) \leq  d_p^*(u,w) +  d_p^*(v,w)$.

(\romannumeral1) Obviously $d_p^*(u,v)\geq 0$.
 Now we show
that
$d_p^*(u,v)=0$ is equivalent to $u=v$.

If $u=v$, then $d_p^*(u,v) = d_p (u,v) =0$.

If $u \not=  v$, then there is an $\al>0$
such that
$[u]_\al \not= [v]_\al$. We claim that outer measure $m^*$ of the set $S :=\{ \beta \in  [0,\al]:  [u]_\beta \not= [v]_\beta \}$
is
greater than $0$. We proceed by contradiction. Suppose $m^* (S) =0$. Then
$[\al-\varepsilon, \al) \not \subseteq  S$ for each $\varepsilon>0$, and therefore
there is a sequence $\{  \al_n, n=1,2,\ldots  \}$
with
$\al_n \in  [\al-\frac{1}{n}, \al)$ and $[u]_{\al_n} =  [v]_{\al_n}$.
This contradicts
with
$[u]_\al \not= [v]_\al$.

Since $m^*(S) >0$, then there exists $k>0$ and $\varepsilon_0>0$
such that
$m^*(\{ H([u]_\bullet, [v]_\bullet) > 1/k\})   > \varepsilon_0$.
So
if $f$ is a measurable function on $[0,1]$ with $f(\cdot)  \geq H([u]_\bullet, [v]_\bullet)$,
then
   $m(f> 1/k) > \varepsilon_0$.
This implies
that
$d_p^*(u,v) >0$.

(\romannumeral2) holds obviously.

(\romannumeral3)
If $f$ is a measurable function on $[0,1]$ with $f(\cdot)  \geq H([u]_\bullet, [w]_\bullet)$
and
 $g$ is a  measurable function on $[0,1]$ with $g(\cdot)  \geq H([v]_\bullet, [w]_\bullet)$,
then
$f+g$ is a  measurable function on $[0,1]$ with $(f+g)(\cdot)  \geq H([u]_\bullet, [v]_\bullet)$.
So
\begin{align*}
 d_p^*(u,v) & \leq   \left(\int_0^1  (f(\al)+g(\al))^p  \,   d\al   \right)^{1/p}
\\
&  \leq   \left(\int_0^1  f(\al)^p  \,   d\al   \right)^{1/p}    +     \left(\int_0^1  g(\al)^p  \,   d\al   \right)^{1/p}.
\end{align*}
From the arbitrariness of $f$ and $g$,
$$ d_p^*(u,v) \leq  d_p^*(u,w) +  d_p^*(v,w). $$

\end{proof}

\end{document}